\documentclass[11pt, reqno]{amsart}
\usepackage{amssymb, amstext, amscd, amsmath, amsthm}
\usepackage{color}
\usepackage{hyperref} 
\usepackage[mathscr]{eucal}

\usepackage{xy}
\xyoption{all}

\usepackage{dsfont}
\theoremstyle{plain}
\newtheorem{theorem}{Theorem}[section]
\newtheorem{thm}[theorem]{Theorem}
\newtheorem{cor}[theorem]{Corollary}
\newtheorem{prop}[theorem]{Proposition}
\newtheorem{lem}[theorem]{Lemma}
\newtheorem{obs}[theorem]{Observation}
\theoremstyle{definition}
\newtheorem{rem}[theorem]{Remark}

\newtheorem{defn}[theorem]{Definition}

\newtheorem{eg}[theorem]{Example}

\DeclareMathAlphabet{\mathpzc}{OT1}{pzc}{m}{it}

\newcommand{\bF}{{\mathbb{F}}}

\newcommand{\bN}{{\mathbb{N}}}

\newcommand{\bZ}{{\mathbb{Z}}}

\newcommand{\bn}{{\mathbf{n}}}
\newcommand{\bs}{{\mathbf{s}}}
\newcommand{\bt}{{\mathbf{t}}}

  \newcommand{\G}{{\mathcal{G}}}

  \newcommand{\N}{{\mathcal{N}}}
\renewcommand{\O}{{\mathcal{O}}}
\renewcommand{\P}{{\mathcal{P}}}
  
  \newcommand{\R}{{\mathcal{R}}}
\renewcommand{\S}{{\mathcal{S}}}

\newcommand{\upchi}{{\raise.35ex\hbox{\ensuremath{\chi}}}}

\newcommand{\qforal}{\quad\text{for all}\quad}

\newcommand{\mini}{\text{min}}

\newcommand{\Aut}{\operatorname{Aut}}

\newcommand{\id}{{\operatorname{id}}}

\newcommand{\Per}{\operatorname{Per}}

\newcommand{\Sub}{\operatorname{Sub}}

\newcommand{\fix}{\operatorname{Fix}}

\newcommand{\Lht}{\operatorname{\Lambda_{ht}}}
\newcommand{\DLht}{\operatorname{\Lambda_{ht}^\prime}}

\newcommand{\ca}{\mathrm{C}^*}

\newcommand{\Fth}{\mathbb{F}_\theta^+}

\newcommand{\mt}{\varnothing}

\begin{document}
\title[Higman-Thompson Like Groups of $k$-Graph C*-Algebras]{Higman-Thompson Like Groups \\ of Higher Rank Graph C*-Algebras}
\author[D. Yang]{Dilian Yang}
\address{Dilian Yang,
Department of Mathematics $\&$ Statistics, University of Windsor, Windsor, ON
N9B 3P4, CANADA}
\email{dyang@uwindsor.ca}

\thanks{The author was partially supported by an NSERC Discovery Grant in Canada.}

\begin{abstract}
Let $\Lambda$ be a row-finite and source-free higher rank graph with finitely many vertices. In this paper, we define the Higman-Thompson like group $\Lht$ of 
the graph C*-algebra $\mathcal{O}_\Lambda$ to be a special subgroup of 
the unitary group in $\mathcal{O}_\Lambda$. It is shown that $\Lht$ is closely related to the topological full groups of the groupoid associated with $\Lambda$. 
Some properties of $\Lht$ are also investigated. 
We show that its commutator group $\DLht$ 
is simple and that $\DLht$ has only one nontrivial uniformly recurrent subgroup if $\Lambda$ is aperiodic and strongly connected. 
Furthermore, if $\Lambda$ is single-vertex, then we prove that $\Lht$ is C*-simple and also provide an explicit description on the stabilizer 
uniformly recurrent subgroup of $\Lht$ under a natural action on the infinite path space of $\Lambda$.

\end{abstract}


\date{}
\maketitle

\section{Introduction}

The Thompson group $V_2$ has been extensively studied since its invention by Richard Thompson in \cite{Tho60} as 
an example of finitely presented infinite simple groups (see also \cite{CFP96, Tho80}). 
Later it is first generalized to $V_{n,r}$ by Higman in \cite{Hig74} (with $V_{2,1}=V_2$). The group $V_{n,r}$ is now known as the Higman-Thompson group. Recently, some higher dimensional analogues have been studied. See, for example, 
\cite{Bri04, LSV20, LV20}. 
In this paper, we introduce a generalization 
 by making full use of some known properties of higher rank graph C*-algebras, and study the properties.  
Our generalization embraces all mentioned above. 
To be a little more precise, for an arbitrary row-finite and source-free higher rank graph $\Lambda$ with finitely many vertices, we associate it a Higman-Thompson like group, 
denoted as $\Lht$, 
which consists of
some special unitaries of the graph C*-algebra $\O_\Lambda$ of $\Lambda$. It turns out that $\Lht$ is isomorphic to the topological full group $F(\G_\Lambda)$
of the groupoid $\G_\Lambda$ associated to $\Lambda$ defined by Nekrashevych  in \cite{Nek19}. Moreover, there is a group homomorphism from 
$\Lht(\cong F(\G_\Lambda))$ to the topological full group $[[\G_\Lambda]]$ defined by Matui in \cite{Mat12, Mat15}, whose kernel is composed of those elements in $\Lht$
coming from equivalent paths of $\Lambda$. In particular, $\Lht$ is not simple if $\Lambda$ is periodic. All of these are given in Section \ref{S:HT}. Some examples are 
exhibited in Section \ref{S:eg}. As an aside, we find that the Thompson groups $V_n$ with $n$ even can be faithfully represented in the automorphism group of an abelian normal subgroup of $\Lht$,
where $\Lambda$ is a single-vertex flip higher rank graph ($k\ge 2$). This seems new in the literature.

When $\Lambda$ is aperiodic, the left regular representation of $\O_\Lambda$ induces a faithful action $\phi$ of $\Lht$ as homeomorphisms on the infinite space $\Lambda^\infty$.
In Section \ref{S:Propo}, we first prove that the commutator group $\DLht$ of $\Lht$ is simple and 
that the stabilizer uniformly recurrent subgroups (URS) is the only nontrivial URS of $\DLht$ if $\Lambda$ is strongly connected. This is heavily facilitated by some recent results in \cite{LBMB18, MB18}. When $\Lambda$ is single-vertex, 
we obtain more properties of $\Lht$:  (1) The action $\phi$ of $\Lht$ on $\Lambda^\infty$ is an extreme boundary action and its stabilizer URS is also explicitly described; 
(2) $\Lht$ is C*-simple and so has the unique trace property; 
(3) The C*-simplicity of $\Lht$ provides another characterization of the aperiodicity of $\Lambda$. 

Let us finish this section by mentioning some connections and differences between this paper and the most recent papers \cite{LSV20, LV20} 
studying the generalizations of the Thompson groups. In terms of the terminologies used here, the paper \cite{LV20} handles the Higman-Thompson like groups associated 
to a special class of aperiodic, row-finite single-vertex higher rank graphs, and our Theorem \ref{T:DG} (i) below recovers the main theorem of \cite{LV20}. 
The paper \cite{LSV20} generalizes \cite{LV20} to aperiodic, cofinal, row-finite, and source-free higher rank graphs with finitely many vertices. The higher dimensional 
generalization of the Thompson group in \cite{LSV20} is defined as the topological full group of the group of units of a Boolean inverse monoid
associated to such a higher rank graph. One can see from Corollary \ref{C:3=} below that ours reconciles theirs in this case. 
The tools of \cite{LSV20, LV20} are mainly from the theory of inverse monoids. Here we take advantage of the known properties of higher rank graph C*-algebras
and take a closer look at the `analytic' properties of Higman-Thompson like groups.

\subsection*{Acknowledgments}
The author is grateful to Matthew Kennedy, Hui Li, and Brita Nucinkis for some helpful discussions.

\section{Prelimiaries} 

This section provides some necessary backgrounds on groupoids, topological full groups of groupoids, higher rank graph algebras and uniformly recurrent subgroups. 
The main resources are from \cite{DY09, KP00, LBMB18, Mat15, Nek19, Rae05, Ren80}. 

\subsection{Groupoids}
A groupoid $\G$ is a small category with inverse. Let $\G^{(0)}$ be the unit space of $\G$;
the mappings $s,r: \G\to \G^{(0)}$ be the source and range maps, respectively; and $\G^{(2)}$ be the set of all comparable pairs in $\G$: 
$\G^{(2)}:=\{(\gamma, \omega)\in \G\times\G: s(\gamma)=r(\omega)\}$. 
A topological groupoid is a groupoid $\G$ equipped with a locally compact topology such that both composition and inversion maps are continuous
and $\G^{(0)}\subseteq \G$ is Hausdorff in the relative topology. 

A subset $U(\subset \G)$ is called a bisection (or $\G$-set) if $r|_U$ and $s|_U$ are injective. An open bisection $U$ of $\G$ induces 
a homeomorphism 
\[
\pi_U: s(U)\to r(U), \  \pi_U:=r\circ (s|_U)^{-1}.
\]
A topological groupoid $\G$ is called \'etale if the range map $r: \G\to \G$ is a local homeomorphism. If $\G$ is \'etale, then has a basis 
consisting of open bisections. 
A groupoid $\G$ is ample if it is \'etale and it has a basis of compact open bisections. 

For $x\in \G^{(0)}$, the isotropy group at $x$ is $x\G x:=\{\gamma\in\G: s(\gamma)=r(\gamma)=x\}$. The isotropy subgroupoid is 
the set $\G':=\cup_{x\in \G^{(0)}} x\G x$. A topological groupoid $\G$ is topologically principal
if $\{x\in \G^{(0)}: x\G x=x\}$ 
is dense in $\G^{(0)}$; it is effective if the interior of $\G'$ is $\G^{(0)}$. If $\G$ is second countable, these two notions are equivalent. 
For $x\in \G^{(0)}$, the set $\G(x):=r(\G x)$ is called the $\G$-orbit of $x$. If every $\G$-orbit is dense in $\G^{(0)}$, $\G$ is said to be minimal. 

\subsection{Topological full groups of groupoids}

Let $\G$ be an \'etale groupoid. Let $B^{\text{co}}(\G)$ be the set of all compact open bisections of $\G$. An element 
$U\in B^{\text{co}}(\G)$ is said to be full if $s(U)=r(U)=\G^{(0)}$.
In \cite{Mat12} Matui defines 
\[
[[\G]]:=\{\pi_U:U\in B^{\text{co}}(\G)\text{ and }U\text{ is full}\}.
\]
He calls $[[\G]]$ the topological full group of $\G$. 
Later, 
Nekrashevych proposes another one in \cite{Nek19} as follows:
\[
F(\G)=\{U\in B^{\text{co}}(\G): U\text{ is full}\}.
\]
It is easy to see that 
\begin{align}
\label{E:theta}
\pi: F(\G)\to [[\G]],\ U\mapsto \pi_U
\end{align}
is a surjective homomorphism, and that $\pi$ is also injective (and so isomorphic) if $\G$ is topologically principal. 

\subsection{Higher rank graph algebras}

Let $1\le k\in \bN$. A higher rank graph (aka $k$-graph or rank-$k$ graph) is a countable small category such that there exists a functor $d:\Lambda \to \mathbb{N}^k$ satisfying the factorization property: for $\mu\in\Lambda, n,m \in \mathbb{N}^k$ with $d(\mu)=n+m$, there exist unique $\alpha,\beta \in \Lambda$ such that  $\mu=\alpha\beta$ with $d(\alpha)=n$, $d(\beta)=m$ and $s(\alpha)=r(\beta)$. Let $(\Lambda_1,d_1)$ and $(\Lambda_2,d_2)$ be two $k$-graphs. A functor $f:\Lambda_1 \to \Lambda_2$ is called a graph morphism if $d_2 \circ f=d_1$.

Let $\Lambda$ be a $k$-graph. For $n\in \bN^k$, let $\Lambda^n:=d^{-1}(n)$.
$\Lambda$ is said to be row-finite (resp.~source-free) if $\vert v\Lambda^{n}\vert<\infty$
(resp.~  $v\Lambda^{n} \neq \mt$ ) for all $v \in \Lambda^0$ and $n \in \mathbb{N}^k$. 
For $\mu,\nu \in \Lambda$, define $\Lambda^{\min}(\mu,\nu):=\{(\alpha,\beta) \in \Lambda \times \Lambda:\mu\alpha=\nu\beta,d(\mu\alpha)=d(\mu)\lor d(\nu)\}$.

Throughout this paper, 
\textsf{all higher rank graphs $\Lambda$ are assumed to be row-finite, source-free, 
and have finitely many vertices.}

The following $k$-graph plays an important role later. 
Let $\Omega_k:=\{(p,q) \in \mathbb{N}^k \times \mathbb{N}^k:p \leq q\}$. For $(p,q), (q,m) \in \Omega_k$, define $(p,q) \cdot (q,m):=(p,m)$, $r(p,q):=(p,p)$, $s(p,q):=(q,q)$, and $d(p,q):=q-p$. Then $(\Omega_k,d)$ is a row-finite source-free $k$-graph. 

\begin{defn}
Let $\Lambda$ be a $k$-graph. Then a family of partial isometries $\{S_\mu\}_{\mu \in \Lambda}$ in a C*-algebra $B$ is called a \emph{Cuntz-Krieger $\Lambda$-family} if
\begin{enumerate}
\item $\{S_v\}_{v \in \Lambda^0}$ is a family of mutually orthogonal projections;
\item $S_{\mu\nu}=S_{\mu} S_{\nu}$ if $s(\mu)=r(\nu)$;
\item $S_{\mu}^* S_{\mu}=S_{s(\mu)}$ for all $\mu \in \Lambda$; and
\item $S_v=\sum_{\mu \in v \Lambda^{n}}S_\mu S_\mu^*$ for all $v \in \Lambda^0, n \in \mathbb{N}^k$.
\end{enumerate}
The C*-algebra $\mathcal{O}_\Lambda$ generated by a universal Cuntz-Krieger $\Lambda$-family $\{s_\mu\}_{\mu \in \Lambda}$ is called the \emph{$k$-graph C*-algebra} of $\Lambda$.
\end{defn}

Let $\Lambda$ be a $k$-graph. Then a graph morphism from $\Omega_k$ to $\Lambda$ is called an infinite path of $\Lambda$. Let $\Lambda^\infty$
denote the set of all infinite paths of $\Lambda$. A $k$-graph $\Lambda$ is called cofinal if, for every $v\in\Lambda^0$ and $x\in \Lambda^\infty$, 
there is $\lambda\in\Lambda$ and $n\in \bN^k$ such that $s(\lambda)=x(0,n)$ and 
$r(\lambda)=v$. 

For $\mu \in \Lambda$,  define the cylinder $Z(\mu):=\{\mu x: x \in \Lambda^\infty,s(\mu)=x(0,0)\}$. Endow $\Lambda^\infty$ with the topology generated by the basic open sets $\{Z(\mu):\mu \in \Lambda\}$. Each $Z(\mu)$ is compact, and $\Lambda^\infty$ is second countable, locally compact, Hausdorff, and totally disconnected. Also $\Lambda^\infty$ is compact as  
$\Lambda^0$ is assumed to be finite. 

Recall from \cite{CKSS14, DY09, Yang15} that there is one equivalence relation $\sim$ on $\Lambda$: 
\[
\mu\sim\nu\in\Lambda \text{ if  } s(\mu)=s(\nu) \text{ and  } \mu x = \nu x \text{ for all }x\in s(\mu)\Lambda^\infty.
\]
When $\Lambda$ has a single vertex, this equivalence yields a periodicity group $\Per\Lambda$ of $\Lambda$. 
We should mention that $\Per\Lambda$ is a subgroup of $\bZ^k$ and that $\Lambda/\!\!\sim$ is a $\textsf{q}(\bN^k)$-graph, where $\textsf{q}$ is the 
natural quotient map $\textsf{q}: \bZ^k\to \bZ^k/\Per\Lambda$. It is worth noticing that this is actually true in a more general setting (\cite{CKSS14, Yang15}).

Finally, we define the path groupoid of $\Lambda$ by
\begin{align*}
\mathcal{G}_\Lambda:=\{(\mu x,d(\mu)-d(\nu),\nu x) \in \Lambda^\infty \times \mathbb{Z}^k \times \Lambda^\infty: s(\mu)=s(\nu),x \in s(\mu)\Lambda^\infty\}.
\end{align*}
For $\mu,\nu \in \Lambda$ with $s(\mu)=s(\nu)$, let $Z(\mu,\nu):=\{(\mu x,d(\mu)-d(\nu),\nu x):x \in s(\mu)\Lambda^\infty\}$. Endow $\mathcal{G}_\Lambda$ with the topology generated by the basic open sets $Z(\mu,\nu)$.
Then $\mathcal{G}_\Lambda$ is an ample groupoid and each $Z(\mu,\nu)$ is a compact open bisection.
It is from \cite{KP00} that $\O_\Lambda$ is isomorphic to the groupoid C*-algebra $\ca(\G_\Lambda)$. 

\subsection{Uniformly recurrent subgroups (URS)}
 Let $G$ be a countable group, and $\Sub(G)$ be the set of all subgroups of $G$. When viewed as a subset of 
$\{0,1\}^G$, the set $\Sub(G)$ is closed for the product topology. The topology induced on $\Sub(G)$ by the product topology is called the 
Chabauty topology, and it makes $\Sub(G)$ a compact space. The action of $G$ on itself by conjugation naturally extends to an action of $G$ on $\Sub(G)$ by homeomorphisms. A uniformly recurrent subgroup (URS) is a closed minimal $G$-invariant subset of $\Sub(G)$. 

Let $G$ be a faithful and minimal
action by homeomorphisms on a compact space $X$. Then this action induces a URS known as the stabilizer URS
of the action $G$ on $X$, which is denoted by $\S_G(X)$ (see \cite{GW15,LBMB18} for more information). 
For $x\in X$, we write $G_x$ as the stabilizer of $x$, and $G_x^0$ as the normal subgroup of $G_x$ which consists of all elements 
$g\in G_x$ such that $g$ fixes pointwise a neighbourhood of $x$.

\section{Higman-Thompson like groups associated to higher rank graphs}
\label{S:HT}

Let $\Lambda$ be a higher rank graph. Let $\G_\Lambda$ be the groupoid associated to $\Lambda$. 
One can check that 
\begin{align}
\label{E:FGLambda}
F(\G_\Lambda)=\left\{U=\bigsqcup_{i=1}^n Z(\alpha_i, \beta_i):
\bigsqcup_{i=1}^n Z(\alpha_i)=\Lambda^\infty=\bigsqcup_{i=1}^n Z(\beta_i),\ n\in \bN
\right\}.
\end{align}
In this case, one can determine the kernel of the homomorphism $\pi: F(\G_\Lambda)\to [[\G_\Lambda]]$ defined in \eqref{E:theta}.

\begin{lem}
\label{L:kertheta}
Let $\pi: F(\G_\Lambda)\to [[\G_\Lambda]]$ be the epimorphism given in \eqref{E:theta}. Then 
\[
\ker\pi=\left\{\bigsqcup_{i=1}^n Z(\alpha_i, \beta_i)\in F(\G_\Lambda): \alpha_i\sim \beta_i \text{ for all }1\le i\le n\right\}. 
\]
\end{lem}

\begin{proof}
The proof follows from the calculations below:
\begin{align*}
&\pi\left(\bigsqcup_{i=1}^n Z(\alpha_i, \beta_i)\right)=\id_{\Lambda^\infty}\\
\iff &\alpha_i x = \beta_i x \text{ for all }x\in s(\alpha_i)\Lambda^\infty(=s(\beta_i)\Lambda^\infty)\quad (1\le i\le n)\\
\iff&\alpha_i\sim \beta_i \quad (1\le i\le n).
\end{align*}
We are done. 
\end{proof}

Notice that $\O_\Lambda$ is unital as $\Lambda^0$ is finite. 
It turns out to be very convenient to connect the topological full group $F(\G)$ with the unitaries of the graph C*-algebra $\O_\Lambda$ of $\Lambda$. 
We need the following lemma first.

\begin{lem}
\label{L:U}
Let $\alpha_i, \beta_i\in \Lambda$ with $s(\alpha_i)=s(\beta_i)$ ($1\le i\le n$). Then $\sum_{i=1}^n s_{\alpha_i} s_{\beta_i}^*$ 
is a unitary in $\O_\Lambda$ if and only if 
\begin{align}
\label{E:CK}
\sum_{i=1}^n s_{\alpha_i} s_{\alpha_i}^*= I=\sum_{i=1}^n s_{\beta_i} s_{\beta_i}^*.
\end{align}
\end{lem}

\begin{proof}
Set $U:=\sum_{i=1}^n s_{\alpha_i} s_{\beta_i}^*$. 

 ``If": First notice that the identities in \eqref{E:CK} imply that $s_{\alpha_i}^* s_{\alpha_j}=0=s_{\beta_i}^* s_{\beta_j}$ for $i\ne j$. 
Indeed, 
\begin{align*}
 \sum_{i=1}^n s_{\alpha_i} s_{\alpha_i}^*= I
&\implies (\sum_{i=1}^n s_{\alpha_i} s_{\alpha_i}^*)s_{\alpha_j}= s_{\alpha_j}\\
&\implies (\sum_{i\ne j} s_{\alpha_i} s_{\alpha_i}^*)s_{\alpha_j}= 0\ (\text{as }s_{\alpha_j} \text{ is a partial isometry})\\
&\implies s_{\alpha_j}^*(\sum_{i\ne j} s_{\alpha_i} s_{\alpha_i}^*)s_{\alpha_j}= 0\\
&\implies \sum_{i\ne j} (s_{\alpha_i}^*s_{\alpha_j})^*(s_{\alpha_i}^*s_{\alpha_j})= 0\\
&\implies s_{\alpha_i}^*s_{\alpha_j}=0\quad \text{ if $i\ne j$}.
\end{align*}
Similarly, one has $s_{\beta_i}^* s_{\beta_j}=0$ for $i\ne j$. 
 Then one has 
 \begin{align*}
 U^* U 
& =\big(\sum_{i=1}^n s_{\beta_i} s_{\alpha_i}^* \big)\big(\sum_{i=1}^n s_{\alpha_j} s_{\beta_j}^* \big)
    =\sum_{i=1}^n s_{\beta_i} s_{\alpha_i}^*s_{\alpha_i} s_{\beta_i}^*\\
& =\sum_{i=1}^n s_{\beta_i} s_{s(\alpha_i)} s_{\beta_i}^*=\sum_{i=1}^n s_{\beta_i} s_{\beta_i}^* \stackrel{\text{by }\eqref{E:CK}}= I.
 \end{align*}
 Similarly, we have $UU^*=I$. 
 
 ``Only if": Suppose that $U=\sum_{i=1}^n s_{\alpha_i} s_{\beta_i}^*$ is a unitary in $\O_\Lambda$. In order to show that $\sum_{i=1}^n s_{\alpha_i} s_{\alpha_i}^*= I$, we can assume that all $\beta_i$'s in the same degree by using the Cuntz-Krieger relations.
In fact, let $m=\vee_{i=1}^n d(\beta_i)$ and then 
 \begin{align*}
 U
 &=\sum_{i=1}^n s_{\alpha_i} \big(\sum_{\gamma_{i_j} \in s(\alpha_i)\Lambda^{m-d(\beta_i)}} s_{\gamma_{i_j}} s_{\gamma_{i_j}}^*\big) s_{\beta_i}^*\\
 &=\sum_{i=1}^n\sum_{i, \, \gamma_{i_j} \in s(\alpha_i)\Lambda^{m-d(\beta_i)}} s_{\alpha_i\gamma_{i_j}}  s_{\beta_i\gamma_{i_j}}^*. 
 \end{align*}
 Thus 
 \begin{align*}
 UU^*=I
 &\implies \sum_{i=1}^n\sum_{\gamma_{i_j} \in s(\alpha_i)\Lambda^{m-d(\beta_i)}} s_{\alpha_i\gamma_{i_j}} s_{\alpha_i\gamma_{i_j}}^*=I\\
 &\implies \sum_{i=1}^n s_{\alpha_i} \big(\sum_{\gamma_{i_j} \in s(\alpha_i)\Lambda^{m-d(\beta_i)}} s_{\gamma_{i_j}} s_{\gamma_{i_j}}^*\big) s_{\alpha_i}^*=I\\
 &\implies \sum_{i=1}^n s_{\alpha_i} s_{\alpha_i}^*= I.
 \end{align*}
The proof of $\sum_{i=1}^n s_{\beta_i} s_{\beta_i}^*=I$ is completely similar. 
\end{proof}

\begin{rem}
In the above, the condition $s(\alpha_i)=s(\beta_i)$ is natural, and just guarantees $s_{\alpha_i} s_{\beta_i}^*\ne 0$ ($1\le i\le n$). So, under this condition,  $\sum_{i=1}^n s_{\alpha_i} s_{\beta_i}^*$ has no zero summands. 
\end{rem}


\begin{cor}
\label{C:Uht}
Let 
\[
\Lht:=\left\{U=\sum_{i=1}^n s_{\alpha_i} s_{\beta_i}^*\in \O_\Lambda: U \text{ is unitary}\right\}.
\] 
Then $\Lht$ is a group in $\O_\Lambda$. 
\end{cor}

\begin{proof}
This is a simple application of Lemma \ref{L:U} and the identity 
$s_\mu^* s_\nu=\sum_{(\alpha, \beta)\in \Lambda^{\mini}(\mu, \nu)}s_\alpha s_\beta^*$.
\end{proof}

We are now ready to introduce the key notion in this note. 

\begin{defn}
\label{D:HT}
The group $\Lht$ in Corollary \ref{C:Uht} is called the \textit{Higman-Thompson like group of $\O_\Lambda$}. 
\end{defn}

It should be mentioned that representing the Thompson group $V_n$ in the Cuntz algebra $\O_n$ (which is the C*-algebra of the graph with one vertex and $n$ edges) was discovered by Nekrashevych in \cite{Nek04}.

\medskip

It turns out that $\Lht$ and $F(\G_\Lambda)$ are isomorphic.

\begin{prop}
\label{P:FGLht}
Keep the notation same as above. Then $\Lht \cong F(\G_\Lambda)$. 
\end{prop}

\begin{proof}
Let us first recall the left regular representation
$L$ of $\O_\Lambda$ on $\ell^2(\Lambda^\infty)$. Suppose that $\{\chi_x: x\in \Lambda^\infty\}$ is the standard orthonormal basis of
$\ell^2(\Lambda^\infty)$. Then $L(s_\mu)=L_\mu$ where 
\[
L_\mu(\chi_x)=\delta_{s(\mu), r(x)}\, \chi_{\mu x}\qforal x\in \Lambda^\infty.
\]
Notice that\footnote{Actually, \eqref{E:L} can be strengthened as follows:  
\[
 \sum_{i=1}^n s_{\alpha_i} s_{\alpha_i}^*= I \iff  \sum_{i=1}^n L_{\alpha_i} L_{\alpha_i}^*= I\iff \bigsqcup_{i=1}^n Z(\alpha_i)=\Lambda^\infty.
\]
To get the first ``$\Longleftarrow$", notice that the augmented left regular representation in \cite[Theorem 4.7.6]{Sim03} is faithful,
 whose restriction onto the diagonal coincides with the left regular representation (although $L$ is not necessarily injective).
} 
\begin{align}
\label{E:L}
 \sum_{i=1}^n s_{\alpha_i} s_{\alpha_i}^*= I \Longrightarrow  \sum_{i=1}^n L_{\alpha_i} L_{\alpha_i}^*= I\Longrightarrow \bigsqcup_{i=1}^n Z(\alpha_i)=\Lambda^\infty.
\end{align}

By Lemma \ref{L:U} and \eqref{E:L} one can a mapping 
$\Psi: \Lht\to F(\G_\Lambda)$ as follows: For $\sum_{i=1}^n s_{\alpha_i} s_{\beta_i}^*\in \Lht$, 
\begin{align*}
\Psi\big(\sum_{i=1}^n s_{\alpha_i} s_{\beta_i}^*\big):=\bigsqcup_{i=1}^n Z(\alpha_i, \beta_i).
\end{align*}
By comparing the multiplications of  $\Lht$ and $F(\G_\Lambda)$, it is not hard to see that the mapping $\Psi$
is a group homomorphism. Clearly by Lemma \ref{L:U} and \eqref{E:L}, $\Psi$ is surjective. 
To see the injectivity of $\Psi$, assume that 
$\Psi\big(\sum_{i=1}^n s_{\alpha_i} s_{\beta_i}^*\big)=\bigsqcup_{i=1}^n Z(\alpha_i, \beta_i)= e_{F(\G_\Lambda)}$ is the identity of $F(\G_\Lambda)$.
Then $\alpha_i=\beta_i=v_i\in \Lambda^0$ for all $1\le i\le n$. So $\bigsqcup_{i=1}^n Z(\alpha_i)=\Lambda^\infty$ implies $\{v_i: 1\le i\le n\} =\Lambda^0$. 
Hence $\sum_{i=1}^n s_{\alpha_i} s_{\beta_i}^*=\sum_{v\in\Lambda^0} s_v s_v^*=\sum_{v\in\Lambda^0} s_v=I$. Therefore $\Psi$ is injective, and so isomorphic.
\end{proof}

\smallskip
\begin{rem}
\label{R:N}
Composing $\pi$ with the mapping $\Psi$ in the proof of Proposition \ref{P:FGLht} gives a homomorphism $\pi\circ \Psi: \Lht\to [[\G_\Lambda]]$. 
Let $\N:=\ker(\pi\circ\Psi)$ and so 
\[
\N=\left\{\sum_i s_{\mu_i} s_{\nu_i}^*\in \Lht: \mu_i\sim \nu_i\right\}
\]
by Lemma \ref{L:kertheta} and Proposition \ref{P:FGLht}. Then, by \cite{BNR14, Yang16}, 
$\N$ is an abelian normal subgroup of $\Lht$. 

It follows from \cite{Yang15} that, when $\Lambda_{\Per}^0=\Lambda^0$, one has $\N\ne \{1_{\Lht}\}$ if and only if $\Lambda$ is periodic 
(refer \cite[Section 4]{Yang15} for the notation $\Lambda_{\Per}^0$). 
\end{rem}

The following corollary is immediate. 

\begin{cor}
\label{C:3=}
Let $\Lambda$ be an aperiodic higher rank graph. Then 
\[
\Lht \cong F(\G_\Lambda)\cong [[\G_\Lambda]].
\]
\end{cor}

So if $\Lambda$ is aperiodic and cofinal such that $\Lambda^\infty$ is a Cantor space  (refer to Subsection \ref{SS:simpD} below), then 
$\Lht$ provides an invariant for for the groupoid $\G_\Lambda$ of $\Lambda$ by \cite[Propositions 4.5, 4.8]{KP00}, \cite[Theorem 3.10]{Mat15}, and Corollary \ref{C:3=}.

The observation below will be used later. 
\begin{obs}
\label{O:Lfaithful}
If $\Lambda$ is aperiodic, 
the left regular representation $L$ of $\O_\Lambda$ induces a natural faithful action $\phi$ of $\Lht$ on the boundary $\Lambda^\infty$:  
\begin{align}
\label{E:phi}
\phi(U)\cdot x := L(U)x \qforal U\in \Lht \text{ and }x\in \Lambda^\infty.
\end{align}  
\end{obs}
In fact, to see the faithfulness of $\phi$, let $U=\sum_{i=1}^n s_{\mu_i}s_{\nu_i}^*\in \Lht$ 
be such that $Ux=x$ for all $x\in \Lambda^\infty$. In particular,  
$U(\nu_i y)=\nu_i y=\mu_i y$ for all $y\in s(\nu_i)\Lambda^\infty$. So $\mu_i\sim\nu_i$. This contradicts the aperiodicity of $\Lambda$ (\cite{Yang15, Yang16}).

For later convenience, we introduce the following definition. 

\begin{defn}
\label{D:ortho}
A family $\{\mu_i\in \Lambda: 1\le i\le n\}$ is called \textit{orthogonal} if $s_{\mu_i}^*s_{\mu_j}=0$ for $i\ne j$. An orthogonal family is 
said to be \textit{complete} if it is defect free: $\sum_{i=1}^n s_{\mu_i} s_{\mu_i}^*=I$. 
\end{defn}

We end this section with a natural generalization of Lemma \ref{L:U}, which could be used in the future. 

\begin{rem}
\label{R:sr}
Let 
\[
\P(\Lambda):=\left\{W=\sum_{i=1}^m s_{\mu_i}s_{\nu_i}^*\in \O_\Lambda: W \text{ is a partial isometry}\right\}\footnote{We always assume that $W$ has no summand $0$. Namely,  $s(\mu_i)=s(\nu_i)$ for $1\le i\le m$.}.
\]
Similar to the proof of Lemma \ref{L:U}, one can see that $W=\sum_{i=1}^m s_{\mu_i}s_{\nu_i}^*\in \P(\Lambda)$, if and only if both 
both $\{\mu_i\}_{i=1}^m$ and $\{\nu_i\}_{i=1}^m$ are orthogonal.

Then there is a bijection $\widetilde \Psi: B^{\text{co}}(\G_\Lambda)\to \P(\Lambda)$ given by 
\[
\widetilde \Psi\big(\bigsqcup Z(\mu_i, \nu_i)\big):= \sum_{i=1}^m s_{\mu_i}s_{\nu_i}^*. 
\]
Notice that both $B^{\text{co}}(\G_\Lambda)$ and $\P(\Lambda)$ are (unital) inverse semigroups. The mapping $\widetilde \Psi$ is actually an inverse semigroup isomorphism. 
Then, for $V=\bigsqcup Z(\mu_i, \nu_i)\in  B^{\text{co}}(\G_\Lambda)$, one can check that 
\begin{center}
$s(V)=Z(\mu)$ for some $\mu\in\Lambda$ $\iff$
$\widetilde \Psi(V)^* \widetilde \Psi(V)=s_\mu s_\mu^*$;\\
$r(V)=Z(\nu)$ for some $\nu\in\Lambda$ $\iff$ 
$\widetilde \Psi(V) \widetilde \Psi(V)^*=s_\nu s_\nu^*$.
\end{center}
\end{rem}

\section{Some Examples}

\label{S:eg}

In this section, we exhibit some examples of Higman-Thompson like groups. In particular, we obtain a seemingly new realization of the Thompson group $V_n$ via (periodic) single-vertex higher rank graphs with the flip relation.

Since single-vertex higher rank graphs play a vital role in some examples below, let us first recall some relevant details from \cite{DPY08, DPY09}. 
Let $\Lambda:=\Fth$ be a single-vertex $k$-graph. Suppose that $\Lambda$ has $n_i$ $i$-th color edges $e_\bs^i$ $(1\le \bs\le n_i,\ 1\le i\le k)$
$\theta_{ij}(\bs,\bt)=(\bs',\bt')$. Then we have the commutation relations
\[
e_\bs^i e_\bt^j=e_{\bt'}^j e_{\bs'}^i \quad (1\le i<j\le k,\ 1\le \bs, \bs'\le n_i, 1\le \bt, \bt'\le n_j).
\]
Thus $\Fth$ is a special (unital) semigroup (\cite{DPY08, DY09}).  

\begin{eg} 
\label{Ex:nV}
Let $\Lambda=\Fth$ be a single-vertex $k$-graph.  
When $\Fth$ is aperiodic, $\Lht$ is nothing but $\mathscr{G}(\Fth)$ 
in \cite{LV20}. 
In particular, if $n_i= 2$ for all $1\le i \le k$ and $\theta_{12}=\id$, one has $\Lht \cong kV$ (the Brin's $k$-dimensional Thompson group \cite{Bri04}). 

\end{eg}

\begin{eg}
\label{Ex:LSV}
As mentioned in the introduction, if $\Lambda$ is aperiodic and cofinal, then it follows from  
\cite[Theorems 4.23 and 5.15]{LSV20},  the beginning of \cite[Section 7]{LSV20}, and Corollary \ref{C:3=} above that $\Lht$ reconciles the 
generalized Thompson group $\mathscr{G}(\Lambda)$ studied in \cite{LSV20}. 
\end{eg}

\begin{eg} Let $\Lambda$ be the following rank-1 graph:
\[
\xygraph{
!{<0cm, 0cm>; <1cm,0cm>:<0cm,1cm>::}
!{(0,0)}*+{v_1 \bullet}="a1"
!{(2,0)}*+{v_2 \bullet}="a2"
!{(4,0)}*+{v_3 \bullet}="a3"
!{(4.7,0)}*+{\dots }="dots"
!{(5.6,0)}*+{v_{r-1}\bullet}="ar-1"
!{(7.6,0)}*+{\bullet v_r}="ar"
"a2":@/^.5cm/^{e_1}"a1"
"a3":@/^.5cm/^{e_2}"a2"
"ar":@/^.5cm/^{e_r}"ar-1"
"a1":@/^1.2cm/_{f_1}"ar"
"a1":@/^1.4cm/^{\vdots \ f_2}"ar"
"a1":@/^2.1cm/^{f_n}"ar"
}
\]
\vskip .3cm
\noindent
Then $\Lht$ is noting but the Higman group $V_{n,r}$ (see \cite{Mat15}). 
\end{eg}

The following property is well-known. 

\begin{lem}
\label{L:abenor}
Let $N$ be an abelian normal subgroup of $G$. Then there is an action $\alpha$ of $G/N$ on $N$ by conjugation: 
\[
\alpha:G/N\to \Aut(N), \ gN\mapsto c_g,
\]
where $c_g(n)=gng^{-1}$ for all $n\in N$. 
\end{lem}

\begin{eg}[\textsf{A new realization of $V_n$ with $n$ even}]
\label{Ex:flip}
Consider the single-vertex rank-2 graph $\Lambda:=\Fth$ with $n(\ge 2)$ edges for each color, where $\theta$ is the flip permutation:  
$\theta(\bs,\bt)=(\bt,\bs)$. In this case, we write $\Fth$ as $\bF^+_{\text{flip}}$.
So 
\[
e_\bs^1 e^2_\bt=e^2_\bs e^1_\bt \quad (1\le \bs, \bt\le n).
\]
\vskip .01cm
$$
\xy (-2.5, 0)*+{v}; (0,0)*+{\bullet}; 
(-10,0)*+{\textcolor{blue}{\xycircle<0.98cm>{}}}; (-15,0)*+{\textcolor{blue}{\xycircle<1.45cm>{}}};
(-16,0)*+{e_1^1}; (-25,0)*+{\cdots}; (-33,0)*+{e_n^1}; 
(-15,15)*+{\textcolor{blue}{>}}; (-10,10)*+{\textcolor{blue}{>}}; 
(10,0)*+{\textcolor{red}{\xycircle<0.98cm>{}}}; (15,0)*+{\textcolor{red}{\xycircle<1.45cm>{}}};
(16,0)*+{e_1^2}; (25,0)*+{\cdots}; (33,0)*+{e_n^2};
(15,15)*+{\textcolor{red}{<}}; (10,10)*+{\textcolor{red}{<}};
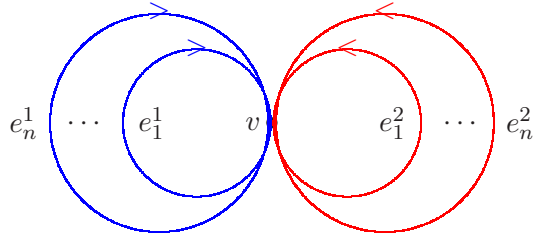
(3, -20)*+{\text{the 1-skeleton of $\Fth$}}
\endxy
$$
\vskip .2cm
\noindent
Then one can show that 
$
\Lht/\ker(\pi\circ\Psi)\cong(\Lambda/\!\!\sim)_{\text{ht}}. 
$
But $\Lambda/\!\!\sim$ is a single-vertex $1$-graph with $n$ edges. Then $(\Lambda/\!\!\sim)_{\text{ht}}\cong V_n$. Thus
\[
\Lht/\ker(\pi\circ\Psi)\cong(\Lambda/\!\!\sim)_{\text{ht}}\cong V_n.
\]

As before, let $\N:=\ker(\pi\circ\Psi)$. From Remark \ref{R:N}, $\N$ is a nontrivial abelian normal subgroup of $\Lht$ as $\bF^+_{\text{flip}}$ is periodic (\cite{DY09}).
Furthermore, the above shows that 
\[
\Lht/\N\cong V_n.
\] 
Then by Lemma \ref{L:abenor} 
there is a homomorphism from $\Lht/\N\cong V_n$  to $\Aut(\N)$. 
But it is well-known that $V_n$ is simple for even $n$. Hence, in this case, $V_n$ can be realized as subgroup of $\Aut(\N)$.
Therefore, let us record the following observation which appears new in the literature. 

\begin{obs}
The Thompson group $V_n$ with $n$ even can be (faithfully) represented in the automorphism group of the abelian 
group $\N$ in the C*-algebra of $\bF^+_{\rm flip}$.
\end{obs}

Note that the above single-vertex flip rank-2 graph can be replaced by any flip single-vertex rank-$k$ graph ($k\ge 2$) with $n$ edges for each color. 
\end{eg}

\begin{eg}
\label{Ex:noname}
Here we provide an example showing that its Higman-Thompson like group is $\bZ$. 
Consider the 1-graph $\Lambda$ determined by the 2-cycle graph: 
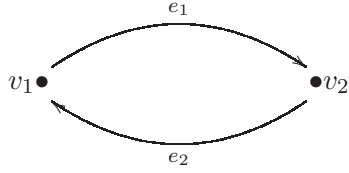
\[
\xygraph{
!{<0cm, 0cm>; <1cm,0cm>:<0cm,1cm>::}
!{(2,0)}*+{v_1 \bullet}="b"
!{(6,0)}*+{\bullet v_2}="e"
"b":@/^.8cm/^{e_1}"e"
"e":@/^.8cm/^{e_2}"b"
}
\]

Then 
\[v_1 \Lambda=\{v_1, e_2e_1, e_2e_1 e_2, (e_2e_1)^2, ...\} \text{ and }v_2 \Lambda=\{v_2, e_1e_2, e_1e_2 e_1, (e_1e_2)^2, ...\}.\] 
If 
$U=s_{\mu_1}s_{\nu_1}^*+s_{\mu_2}s_{\nu_2}^*\in \Lht$, then either $\{\mu_1, \mu_2\}=\{(e_2e_1)^m e_2, (e_1e_2)^me_1\}$ or 
$\{\mu_1, \mu_2\}=\{(e_2e_1)^m , (e_1e_2)^m\}$ for some $m\in \bN$, and similarly for $\{\nu_1, \nu_2\}$. 
So we have the identities 
\begin{align*}
s_{(e_2e_1)^me_2}+s_{(e_1e_2)^me_1}&=(s_{e_1}+s_{e_2})^{2m+1},\\
s_{(e_2e_1)^m}+s_{(e_1e_2)^m}&=(s_{e_1}+s_{e_2})^{2m}.
\end{align*}
Then one can show that $\Lht=\langle s_{e_1}+s_{e_2}\rangle \cong \bZ$.  

Clearly, $\Lambda^\infty=\{(e_1e_2)^\infty, (e_2e_1)^\infty\}$ and so it is not a Cantor space. Also, the adjacency matrix of $\Lambda$ is a permutation matrix.

Notice that $e_1e_2\sim v_2$ and $e_2 e_1\sim v_1$. Thus $\ker(\pi\circ\Psi)=\{(s_{e_1}+s_{e_2})^{2m}: m\in\bZ\}\cong 2\bZ$, and so 
$\Lht/\ker(\pi\circ\Psi)\cong \bZ_2$, which is isomorphic to the Higman-Thompson like group of the quotient graph $\Lambda/\!\!\sim$, which is a $\bZ_2$-graph. 
\end{eg}

The above example can be generalized to any $n$-cycle graph with $n\ge 1$. 

\smallskip
\begin{rem}
\label{Ex:gen} 
This remark sketches a common framework for some periodic $k$-graphs (\cite{CKSS14, Yang15, Yang16}), which generalizes 
Examples \ref{Ex:flip} and \ref{Ex:noname}.
Suppose that the periodicity group $\Per\Lambda$ of $\Lambda$ is a subgroup of $\bZ^k$. Under some conditions,  
one could prove that $\Lht/\ker(\pi\circ\Psi)\cong (\Lambda/\!\!\sim)_{\text{ht}}$, where 
$\Lambda/\!\!\sim$ is a $\textsf{q}(\bN^k)$-graph with $\textsf{q}$ being the natural quotient map $\textsf{q}: \bZ^k\to \bZ^k/\Per\Lambda$. 
\end{rem}

\section{Some properties of $\Lht$}

\label{S:Propo}

Let $\Lambda$ be an aperiodic strongly connected $k$-graph. In this section, we show the simplicity of the commutator group $\DLht$ and 
$\DLht$ has only one nontrivial URS. 
Furthermore, if $\Lambda$ is single-vertex, then we prove that $\Lht$ is C*-simple and also provide an explicit description on the stabilizer URS of $\Lht$ under 
a natural action on the infinite path space. 

For $\mu\in\Lambda$, let $P_\mu:=s_\mu s_\mu^*$ and $P_\mu^\perp = I-P_\mu$.

\subsection{The simplicity of the commutator group $\DLht$}
\label{SS:simpD}

Let $\Lambda$ be a (finite) strongly connected $k$-graph $\Lambda$. A sufficient condition is given in \cite[Proposition 2.17]{FGJKP20} to guarantee that $\Lambda^\infty$ is a Cantor space (i.e., a compact, metrizable, totally disconnected space with no isolated points). In particular, if every vertex receives at least two $i$-th coloured edges for every $1\le i\le k$, then $\Lambda^\infty$ is a Cantor space (see \cite[Remark 2.18]{FGJKP20}). 

Recall that a minimal groupoid whose unit space is a Cantor space is called \textit{purely infinite} if, for every clopen set $A\subseteq \G^{(0)}$, there exist compact open bisections 
$U,V\subseteq \G$ such that $s(U)=s(V)=A$ and $r(U)\sqcup r(V)\subseteq A$.

\smallskip
Motivated by \cite[Lemma 6.1]{Mat15}, we have the following. 

\begin{lem}
\label{L:pureinf}
Let $\Lambda$ be a strongly connected $k$-graph such that $\Lambda^\infty$ is a Cantor space. Then $\G_\Lambda$ is minimal and purely infinite.
\end{lem}

\begin{proof}
To show this,  we 
apply \cite[Proposition 4.11 (2)]{Mat15}. For this, let $A\ne \mt\ne B$ be two clopen subsets of $\Lambda^\infty$, Write $A=\bigsqcup_{i=1}^m Z(\mu_i)$ and $B=\bigsqcup_{j=1}^n Z(\nu_j)$. 
By dividing up $Z(\nu_j)$'s if necessary, we assume that $m\le n$. Let $f:\{1,\ldots,m\}\to \{1,\ldots, n\}$ be an injection. Since $\Lambda$ is strongly connected, there is $\lambda_i \in s(\nu_{f(i)})\Lambda s(\mu_i)$. 
Let $\mu_i':=\nu_{f(i)} \lambda_i$. Then $Z(\mu_i')\subset Z(\nu_{f(i)})$. Let $U:=\bigsqcup_{i=1}^m Z(\mu_i', \mu_i)$. Then $U$ is a compact open bisection satisfying $s(U)=A$ and $r(U)\subseteq B$.  
\end{proof}

Recall that if $\Lambda$ is aperiodic, then by Observation \ref{O:Lfaithful} the left regular representation of $\O_\Lambda$ induces a faithful action $\phi$ of $\Lht$ on the infinite path 
space $\Lambda^\infty$ (see \eqref{E:phi}). So one has the stabilizer URS $\S_{\Lht}(\Lambda^\infty)$ of the action $\phi$.
This yields the stabilizer URS $\S_{\DLht}(\Lambda^\infty)$ of the induced action $\DLht \curvearrowright\Lambda^\infty$.

\begin{thm}
\label{T:DG}
Let $\Lambda$ be an aperiodic strongly connected $k$-graph such that $\Lambda^\infty$ is a Cantor space.  
\begin{itemize}
\item[(i)] If $N$ is a nontrivial normal subgroup of $\Lht$, then $\DLht\subseteq N$. In particular $\DLht$ is simple. 

\item[(ii)] The only nontrivial URS of $\DLht$ is the stabilizer URS. 
\end{itemize}
\end{thm}

\begin{proof}
Since $\Lambda$ is aperiodic and strongly connected, it is known that $\G_\Lambda$ is minimal and topologically principal (see, e.g., \cite[Propositions 6.3 and 6.5]{LY-IMRN}). 

(i). This immediately follows from Proposition \ref{P:FGLht}, Lemma \ref{L:pureinf} and \cite[Theorem 4.16]{Mat15}.

(ii). By Proposition \ref{P:FGLht}, Lemma \ref{L:pureinf} and \cite[Section 4]{Nek19}, the commutator group $\DLht$ coincides with the alternative subgroup $A(\Lht)$
of $\Lht$. Then \cite[Corollary 6.5]{MB18} ends the proof. 
\end{proof}

\begin{rem}
(a) Theorem \ref{T:DG} (i) generalizes a recent result in \cite{LV20} (cf. Example \ref{Ex:nV}), and also greatly simplifies the proof.

(b) By \cite[Theorem 5.2]{Mat15}, the index mapping $I: [[\G_\Lambda]]\to H_1(\G_\Lambda)$ defined there is also surjective.  
\end{rem}

The following corollary is immediate. 

\begin{cor}
Let $\Lambda$ be an aperiodic strongly connected $k$-graph such that $\Lht$ is simple. Then its URS's are $1_{\Lht}$, $\Lht$ and $\S_{\Lht}(\Lambda^\infty)$.
\end{cor}

When $\Lambda$ has a single-vertex, $\S_{\Lht}(\Lambda^\infty)$ will be explicitly described in Theorem \ref{T:staURS} below.

\subsection{A description of the stabilizer URS}

Let $\Lambda$ be a single-vertex $k$-graph with $n_i$ edges of color $i$ with $n_i\ge 1$. Put 
\[
\bn:=(n_1,\ldots, n_k).
\] 

We first present a technical but explicit extension lemma.

\begin{lem}
\label{L:munu}
Suppose $\mu_i, \nu_i\in\Lambda$ ($1\le i\le n$) such that both $\{\mu_1, \ldots, \mu_n\}$ and 
$\{\nu_1, \ldots, \nu_n\}$ are orthogonal.  Then there is a unitary $U\in \Lht$ such that 
\[
U=\sum_{i=1}^n s_{\mu_i} s_{\nu_i}^* + \sum_{j=1}^l s_{\alpha_j} s_{\beta_j}^*.
\] 
for some $\alpha_j$'s and $\beta_j$'s in $\Lambda$. 
\end{lem}

\begin{proof}
By Proposition \ref{P:FGLht}, it is equivalent to show that there are $\alpha_j, \beta_j\in \Lambda$ ($j=1,\ldots, l$) such that  
\[
\left(\bigsqcup_{i=1}^n Z(\mu_i, \nu_i)\right)\bigsqcup\left(\bigsqcup_{j=1}^l Z(\alpha_j, \beta_j)\right)\in F(\G_\Lambda).
\]

We first extend $\{\mu_1,\ldots, \mu_n\}$ and $\{\nu_1,\ldots, \nu_n\}$ to two complete (orthogonal) families as follows. For simplicity, let $\displaystyle\mathsf{M}:=\sum_{i=1}^n d(\mu_i)$, 
$\displaystyle\mathsf{N}:=\sum_{i=1}^n d(\nu_i)$, $\displaystyle\bs:=\bn^{\mathsf{M}}-\sum_{i=1}^n\bn^{\mathsf{M}-d(\mu_i)}$, and
$\displaystyle\bt:=\bn^{\mathsf{N}}-\sum_{i=1}^n\bn^{\mathsf{N}-d(\nu_i)}$. Set
\begin{align*}
\Lambda^{\mathsf{M}}\setminus \bigcup_{i=1}^n\mu_i\, \Lambda^{\mathsf{M}-d(\mu_i)}
=\{\omega_i: 1\le i\le \bs\}, \quad 
\Lambda^{\mathsf{N}}\setminus \bigcup_{i=1}^n\nu_i\, \Lambda^{\mathsf{N}-d(\nu_i)}
=\{\gamma_i: 1\le i\le \bt\}. 
\end{align*}
Then one can see that 
\begin{align*}
\{\mu_1,\ldots, \mu_n,\omega_1,\ldots, \omega_\bs\}
\end{align*}
and
\begin{align*}
\{\nu_1,\ldots, \nu_n,\gamma_1,\ldots, \gamma_\bt\}
\end{align*}
are two complete families extended from $\{\mu_1,\ldots, \mu_n\}$ and $\{\nu_1,\ldots, \nu_n\}$, respectively. 

Now we adjust the two complete families obtained above to two ones \textsf{with equal length but keep $\mu_i$'s and $\nu_i$'s $(1\le i\le n)$ unchanged}: 
\begin{align}
\label{E:mu}
&\mu_1,\ldots, \mu_n,\omega_1\Lambda^{\mathsf{M}-d(\mu_1)},\ldots, \omega_n\Lambda^{\mathsf{M}-d(\mu_n)}, \omega_{n+1}\Lambda^{\mathsf{N}}, \ldots, \omega_\bs; \\
\label{E:nu}
&\nu_1,\ldots, \nu_n,\gamma_1\Lambda^{\mathsf{N}-d(\nu_1)},\ldots, \gamma_n\Lambda^{\mathsf{N}-d(\nu_n)}, \omega_{n+1}\Lambda^{\mathsf{M}}, \ldots, \gamma_\bt. 
\end{align}
One can easily compute that 
\eqref{E:mu} has $\displaystyle n+\bs -n-1+ \sum_{i=1}^n \bn^{\mathsf{M}-d(\mu_i)}+\bn^{\mathsf{N}}=\bn^{\mathsf{M}}+\bn^{\mathsf{N}}-1$, and 
\eqref{E:nu} has $\displaystyle n+\bt -n-1+ \sum_{i=1}^n \bn^{\mathsf{N}-d(\nu_i)}+\bn^{\mathsf{M}}=\bn^{\mathsf{N}}+\bn^{\mathsf{M}}-1$. So \eqref{E:mu} and \eqref{E:nu} 
have equal number of elements. 

Therefore we obtain a desired extension. 
\end{proof}

In what follows, we show that $\phi:G\curvearrowright \Lambda^\infty$ is an extreme boundary action, that is, $\Lambda^\infty$ is compact, $\phi$ is minimal and extremely proximal 
(see \cite[Subsection 2.1]{LBMB18} for all relevant definitions).

\begin{lem}
\label{L:extreb}
Let $\Lambda$ be an aperiodic single-vertex $k$-graph. Then the action $\phi$ of $\Lht$ on $\Lambda^\infty$ is an extreme boundary action. 
\end{lem}

\begin{proof}
Since $\Lambda$ is aperiodic, it follows from \cite{DY09} and the beginning of Subsection \ref{SS:simpD} that $\Lambda^\infty$ is compact as it is a Cantor space. 

To see the minimality of $\phi$, fix $x\in \Lambda^\infty$ and take $y\in \Lambda^\infty$. Let $n\in \bN$ big enough. Consider $n\textbf{1}=(n,\ldots, n)\in \bN^k$.
Consider the pair $(y(0,n\textbf{1}),  x(0,n\textbf{1}))$. It follows from Lemma \ref{L:munu} that there is $g\in \Lht$ such that $g\cdot  x(0,n\textbf{1})= y(0,n\textbf{1})$. 
This shows that $g\cdot x$ and $y$ are close enough. Therefore $\phi$ is minimal. 

It remains to show that $\phi$ is extremely proximal. For this, we show that each $Z(\mu)$ is compressible. Let $x\in \Lambda^\infty$ and write $x=\lambda y$ with $d(\mu)=d(\lambda)$. Let $N$ be a neighbourhood of $x$. WLOG we assume that $N=Z(\lambda\alpha)$ for some $\alpha\in\Lambda$ (that is $x=\lambda\alpha z$ for some $z\in \Lambda^\infty$). Let $g=s_{\lambda\alpha} s_\mu^*+h\in \Lht$ obtained from Lemma \ref{L:munu}. Then $g\cdot Z(\mu)\subseteq N$. Thus $Z(\mu)$ is compressible. 

For a general closed subset $Y\subsetneq\Lambda^\infty$, say $Y=\bigsqcup_{i=1}^n Z(\mu_i)$. Let $x\in \Lambda^\infty$ and write $x=\lambda y$ with $d(\mu)=d(\lambda)$. Let $N$ be a neighbourhood of $x$. WLOG we assume that $N=Z(\lambda\alpha)$ for some $\alpha\in\Lambda$. Then there are $\nu_i$ ($1\le i\le n$) of the same degree which are mutually orthogonal such that $\bigsqcup_{i=1}^n Z(\lambda\alpha\nu_i)\subseteq Z(\lambda\alpha)$. By Lemma \ref{L:munu} there is $g\in \Lht$ such that $g\cdot Y\subseteq N$, as required. 

 Therefore the action is extremely proximal. 
\end{proof}

Similar to the proof of \cite[Proposition 4.12]{LBMB18} one has 

\begin{thm}
\label{T:staURS}
Let $\Lambda$ be an aperiodic single-vertex $k$-graph. Then the action $\phi$ of $\Lht$ on $\Lambda^\infty$ has Hausdorff germs. 
So its stabilizer URS is given by $\S_{\Lht}(\Lambda^\infty)=\{(\Lht)_x^0: x\in \Lambda^\infty\}$.
\end{thm}

\begin{proof}
Let $x\in\Lambda^\infty$ and $g=\sum_{i=1}^n s_{\alpha_i}s_{\beta_i}^*\in \Lht$ of the reduced form. Suppose that $g\in (\Lht)_x$. Since $\Lambda$ is aperiodic, we have either $\alpha_i=\beta_i$, or 
$(\alpha_i, \beta_i)$ is not an equivalent pair.
So if $\fix(g)$ contains a neighbourhood of $x$, then there is $1\le i\le n$ such that $\alpha_i=\beta_i$ and $x\in Z(\alpha_i)$.  
Hence the interior of $\fix(g)$ is the (disjoint) union of $Z(\alpha)$ where $s_{\alpha} s_\alpha^*$ is a summand of $g$.
Now suppose that $\fix(g)$ does not contain a neighbourhood of $x$, then there is $1\le i\le n$ such that $x\in Z(\beta_i)\cap Z(\alpha_i)$ but $\alpha_i\ne \beta_i$.  It now easily follows that if $\fix(g)$ does not contain a neighbourhood of $x$, then $x$ is not an accumulation point of the interior of $\fix(g)$. Therefore, the action has Hausdorff germs. The assertion of the lemma now follows from Lemma \ref{L:extreb} and \cite[Proposition 2.10]{LBMB18}. 
\end{proof}

\subsection{The C*-simplicity of $\Lht$} 

Let $\Lambda$ be an aperiodic single-vertex $k$-graph again. In this subsection, we will show that $\Lht$ is C*-simple (i.e., its reduced C*-algebra is simple). Recall from Observation
\ref{O:Lfaithful} that $\Lht$ acts on $\Lambda^\infty$ as homeomorphisms in our setting. 

As in \cite{Nek19}, for $\mu\in\Lambda$, define $I_\mu: \Lht\to \Lht$ as follows: 
\[
I_\mu(A)=s_\mu  A s_\mu^*+I-s_\mu s_\mu^*=s_\mu  A s_\mu^*+P_\mu^\perp\qforal A\in \Lht.
\]

\begin{lem}
For any $\mu \in \Lambda$, one has $\Lht\cong I_\mu(\Lht)$. 
\end{lem}

\begin{proof}
Using the fact that $s_\mu^* s_\mu=I$ for all $\mu\in\Lambda$, one can easily check that $\varphi$ is an injective group homomorphism.
To see its surjectivity, let $A\in V_\mu$. If $I_\mu(B)=s_\mu B s_\mu^*+I-s_\mu s_\mu^*=A$, then this forces $B=s_\mu^* A s_\mu$. 
Using the fact that $A P_\mu = P_\mu A P_\mu$, one can verify that $B$ is a unitary (so in $\Lht$) and $I_\mu (B)=A$. 
\end{proof}

For a subset $X\subseteq \Lambda^\infty$, by $\R_{\Lht}(X)$ we denote the \textit{rigid stabilizer of $X$}, i.e., $\R_{\Lht}(X)=\{U\in \Lht: Ux=x \text{ for every }x\in \Lambda^\infty\setminus X\}$. 

\begin{lem}
\label{L:Lhtmu}
$I_\mu(\Lht) = \R_{\Lht}(Z(\mu))$. 
\end{lem}

\begin{proof}
To show $I_\mu(\Lht) \subseteq \R_{\Lht}(Z(\mu))$, given any $U\in \Lht$, one can see that 
$I_\mu(U)s_\nu=s_\nu$ for all $\nu \in \Lambda$ with $d(\mu)=d(\nu)$ but $\mu \ne \nu$. Thus $I_\mu(U)x=x$ for all $x\in \Lambda^\infty\setminus Z(\mu)$,
and so $I_\mu(U)\in \R_{\Lht}(Z(\mu))$. 

For $\R_{\Lht}(Z(\mu))\subseteq I_\mu(\Lht)$, let $U\in \R_{\Lht}(Z(\mu))$. Then $UP_\mu = P_\mu U$, 
which follows from that $UP_\mu^\perp = P_\mu^\perp U P_\mu^\perp$ and that $\R_{\Lht}(Z(\mu))$ is a group. 
Let $f:=s_\mu^* U s_\mu$. Then it is easy to check that $f\in \Lht$ and
Then $I_\mu(f)= s_\mu(s_\mu^*Us_\mu)s_\mu^*+P_\mu^\perp =P_\mu U P_\mu+P_\mu^\perp = U P_\mu+U P_\mu^\perp =U$ as $U P_\mu^\perp = P_\mu^\perp$. 
Therefore $\R_{\Lht}(Z(\mu))\subseteq I_\mu(\Lht)$. 
\end{proof}

It is easy to check that $s_\mu^* s_\nu=0$ if and only if $Z(\mu)\cap Z(\nu)=\mt$ as $\mu$ and $\nu$ have no common extensions. Using this, one can easily verify the following:
If $\mu,\nu\in\Lambda$ such that $s_\mu^* s_\nu=0$, then $I_\mu(\Lht)\cap I_\nu(\Lht)=\{I\}$, and 
\[
I_\mu(A)I_\nu(B)=I_\nu(B)I_\mu(A)=s_\mu A s_\mu^*+s_\nu B s_\nu^*+I-(P_\mu+P_\nu)\qforal A, B\in \Lht. 
\]
Clearly, $I_\mu(A)I_\nu(B) (I-P_\mu-P_\nu)=I-P_\mu-P_\nu$. So $I_\mu(\Lht) I_\nu(\Lht)\subseteq \R_{\Lht}(Z(\mu)\sqcup Z(\nu))$.
In general, we have 

\begin{lem}
If  $\{\mu_i\in \Lambda: 1\le i\le n\}$ be an orthogonal family. Then 
\[
\prod_{i=1}^n I_{\mu_i}(\Lht)\subseteq \R_{\Lht}\left(\bigsqcup_{i=1}^n Z(\mu_i)\right).
\] 
\end{lem}

Notice that $I_{\mu_i}(\Lht)$'s on the left hand side of the above identity are commutative, and so the notation $\prod_{i=1}^n I_{\mu_i}(\Lht)$ makes sense. 

Clearly, $I_{\mu_i}(\Lht)\le \prod_{i=1}^n I_{\mu_i}(\Lht)$. But $I_{\mu_i}(\Lht)\cong \Lht$ and so $I_{\mu_i}(\Lht)$ is non-amenable 
(as $\Lht$ contains the Thompson group $V_n$). Thus $\prod_{i=1}^n I_{\mu_i}(\Lht)$ is non-amenable, and hence $\R_{\Lht}\left(\bigsqcup_{i=1}^n Z(\mu_i)\right)$ is 
non-amenable. Therefore, by the first part of \cite[Theorem 3.7]{LBMB18} and \cite[Corollary 4.3]{BKKO17} we have the following. 

\begin{thm}
\label{T:LhtC*S}
Let $\Lambda$ be an aperiodic single-vertex $k$-graph. Then $\Lht$ is C*-simple and so has the unique trace property. 
\end{thm}

\begin{rem}
Theorem \ref{T:LhtC*S} can be deduced directly from the second part of \cite[Theorem 3.7]{LBMB18} and \cite[Theorem 4.5 (i)]{LBMB18}.
\end{rem}

We finish the note by the following characterizations of the C*-simplicity of $\Lht$. 

\begin{cor} Let $\Lambda$ be a single-vertex $k$-graph. Then the following statements are equivalent. 
\begin{itemize}
\item[(a)] $\Lambda$ is aperiodic. 
\item[(b)] $\O_\Lambda$ is simple. 
\item[(c)] $\Lht$ is C*-simple. 
\end{itemize}
\end{cor}

\begin{proof}
(a)$\iff$(b): See \cite[Theorem 11.1]{HLRS15} (cf. \cite{Yang12}). 

(a)$\implies$(c): This is from Theorem \ref{T:LhtC*S}. 

(c)$\implies$(a): Since $\Lambda$ is single-vertex, obviously $\Lambda^0_{\Per}=\Lambda^0$. 
If $\Lambda$ is periodic, then $\Lht$ has a nontrivial abelian normal subgroup (see Remark \ref{R:N}). Thus $\Lht$ can not be C*-simple. 
\end{proof}

\end{document}